\documentclass[11pt]{amsart}
\usepackage{amsfonts, amsmath, amssymb, amscd, amsthm, graphicx,enumerate}
\usepackage{epsfig, color}
\usepackage{hyperref, stmaryrd}

\makeatletter
\def\blfootnote{\xdef\@thefnmark{}\@footnotetext}
\makeatother

\newtheorem{thm}{Theorem}
\newtheorem{cor}[thm]{Corollary}
\newtheorem{lem}[thm]{Lemma}

\newfont{\eufm}{eufm10}

\renewcommand{\phi}{\varphi}

\newcommand{\Z}{\mathbb Z}

\def\fix{\mathrm{Fix}}
\def\axis{\mathrm{Axis}}
\def\stab{\mathrm{Stab}}

\def\mc {\mathcal}

\begin{document}

\title{Splittings of right-angled Artin groups}

\author{M. Hull}
\address{}
\email{}

\date{}
\maketitle

\begin{abstract}
We show that if a right-angled Artin group $A(\Gamma)$ has a non-trivial, minimal action on a tree $T$ which is not a line, then $\Gamma$ contains a separating subgraph $\Lambda$ such that $A(\Lambda)$ stabilizes an edge in $T$.
\end{abstract}

\section{Introduction}
Given a graph $\Gamma$, let $A(\Gamma)$ denote the associated \emph{right-angled Artin group}. That is, $A(\Gamma)$ is the group given by the presentation
\[
\langle V(\Gamma)\;|\; [u, v]=1 \text{ whenever } (u, v)\in E(\Gamma)\rangle.
\]

A natural question in the study of right-angled Artin groups is how are the graph-theoretic properties of $\Gamma$ related to group-theoretic properties of $A(\Gamma)$? We consider this question in the context of splittings of the group $A(\Gamma)$. Here by a splitting of a group $G$ we mean a graph of groups decomposition of $G$ and by a splitting over a subgroup $H$ we mean a graph of groups decomposition where $H$ is an edge group. By Bass-Serre theory, for any such splitting there is an action of $G$ on a tree\footnote{All trees in this paper are assumed to be simplicial trees.} (called the \emph{Bass-Serre tree} of the splitting) and conversely any action of $G$ on a tree has a corresponding splitting where the edge groups are stabilizers of edges in the tree.

 It is well-known that $A(\Gamma)$ splits as a non-trivial free product if and only if $\Gamma$ is disconnected. Clay showed that $A(\Gamma)$ splits over $\mathbb Z$ if and only if $\Gamma$ contains a cut-vertex \cite{C}. Groves and the author generalized Clay's result to show that unless $\Gamma$ is a complete graph, $A(\Gamma)$ splits over an abelian subgroup if and only if $\Gamma$ contains a cut-clique, that is a complete subgraph $\Lambda$ such that $\Gamma\setminus\Lambda$ is disconnected \cite{GH}.

One direction of these implications is straightforward and follows from the general observation that if $\Lambda$ is a separating subgraph of $\Gamma$, that is a subgraph such that $\Gamma\setminus \Lambda$ is disconnected, then $A(\Gamma)$ splits over $A(\Lambda)$ as an amalgamated product
\[
A(\Gamma)\cong A(\Gamma_1\cup\Lambda)\ast_{A(\Lambda)}A(\Gamma_2\cup\Lambda)
\]
where $\Gamma_1$ is a connected component of $\Gamma\setminus\Lambda$ and $\Gamma_2=\Gamma\setminus(\Gamma_1\cup\Lambda)$. 

Another way to construct splittings of $A(\Gamma)$ is to consider actions of $A(\Gamma)$ on a line. Any such action will produce a splitting of $A(\Gamma)$ as an HNN-extension over the kernel of the action. When $\Gamma$ is connected the set of such actions is equivalent to the set homomorphisms from $A(\Gamma)\to\Z$. For example, $A(\Gamma)$ splits as an HNN-extension over the Bestvina-Brady subgroup which is the kernel of the homomorphism $A(\Gamma)\to\Z$ defined by sending each generator of $A(\Gamma)$ to $1$. Homomorphisms $A(\Gamma)\to\Z$ all factor through the abelianization of $A(\Gamma)$ which is isomorphic to $\Z^n$ (where $n$ is the number of vertices of $\Gamma$). Hence the set of homomorphisms $\mathbb Z^n\to \Z$ gives a parameterization of these types of splittings of $A(\Gamma)$.  

We now consider splittings of $A(\Gamma)$ where the corresponding Bass-Serre tree is not a line. An action of a group $G$ on a tree $T$ is called \emph{minimal} if $T$ has no proper $G$--invariant subtrees and \emph{non-trivial} if there is no point of $T$ which is fixed by all elements of $G$.

\begin{thm}\label{mainthm}
Suppose $A(\Gamma)$ has a non-trivial minimal action on a tree $T$ which is not a line. Then $\Gamma$ has an induced subgraph $\Lambda$ such that $\Gamma\setminus\Lambda$ is disconnected and $A(\Lambda)$ stabilizes an edge in $T$.
\end{thm}

We note that the edge group in a given splitting of $A(\Gamma)$ may be strictly larger than the subgroup $A(\Lambda)$ produced in the above theorem. For example, if there is an epimorphism $f\colon A(\Gamma_1)\to A(\Gamma_2)$ and $\Lambda_2$ is a separating subgraph of $\Gamma_2$, then the splitting of $A(\Gamma_2)$ over $A(\Lambda_2)$ induces a splitting of $A(\Gamma_1)$ over $f^{-1}(A(\Lambda_2))$. The subgraph $\Lambda_1$ constructed in Theorem \ref{mainthm} will consist of those vertices of $\Gamma_1$ which $f$ maps into $A(\Lambda_2)$. As long as some element of $\ker(f)$ does not belong to $A(\Lambda_1)$, $A(\Lambda_1)$ will be a proper subgroup of $f^{-1}(A(\Lambda_2))$.

Nevertheless, it is common to consider all possible splittings over a particular family of subgroups, for example the family of all abelian subgroups. If $\mc{A}$ is any family of subgroups of a group $G$, then a splitting of $G$ is called an $\mc{A}$--splitting if all edge groups of the splitting belong to $\mc{A}$.
\begin{cor}
Let $\Gamma$ be a connected graph and let $\mc{A}$ be a family of subgroups of $A(\Gamma)$ such that $\mc{A}$ is closed under taking subgroups. If $A(\Gamma)$ has a non-trivial $\mc{A}$--splitting then either there exists a homomorphism $\phi\colon\ A(\Gamma)\to\Z$ with $\ker(\phi)\in\mc{A}$ or $\Gamma$ has a separating subgraph $\Lambda$ such that $A(\Lambda)\in\mc{A}$.
\end{cor}

\noindent{\bf Acknowledgements.}  The author thanks Talia Fern\'os for asking what would happen if we considered splittings over non-abelian subgroups and for helpful comments.
\section{Proof}
The proof of Theorem \ref{mainthm} only uses elementary properties of group actions on trees and right-angled Artin groups which we now describe.

Let $G$ be a group acting on a tree $T$. For an element $g\in G$, let $\fix(g)$ denote the set $\{x\in T\;|\; gx=x\}$. If $\fix(g)\neq \emptyset$, the $g$ is called \emph{elliptic}. In this case $\fix(g)$ is connected and hence a subtree of $T$. If $g\in G$ is not elliptic, then $g$ is \emph{hyperbolic} which means that  $T$ contains a unique line which is fixed by $g$ set-wise and on which $g$ acts as a non-trivial translation. In this case the corresponding line is called the \emph{axis} of $g$ which we denote by $\axis(g)$. If $g$ is hyperbolic and $h$ is any element of $G$, then $h^{-1}gh$ is hyperbolic with axis $h(\axis(g))$. In particular, if $g$ and $h$ commute, then $h$ fixes the axis of $g$ set-wise. The next lemma follows easily from this observation.

\begin{lem}\label{L1}
Suppose that a group $G$ acts on a tree $T$ and that $g,h \in G$ are commuting elements with $g$ acting hyperbolically on $T$.
\begin{enumerate}
\item If $h$ is elliptic, then $\axis(g)\subseteq \fix(h)$.
\item If $h$ is hyperbolic, then $\axis(g)=\axis(h)$.
\end{enumerate}
\end{lem}

For commuting elliptic elements we use the following.
\begin{lem} \label{l:fix intersect}\cite[Lemma 1.1]{GH}
 Suppose that a group $G$ acts on a tree $T$ and that $g,h \in G$ are commuting elements which both act elliptically on $T$.  Then $\mathrm{Fix}(g) \cap \mathrm{Fix}(h) \ne \emptyset$.
\end{lem}


\begin{proof}[Proof of Theorem \ref{mainthm}]
We identify vertices of $\Gamma$ with generators of $A(\Gamma)$ in the natural way. First we suppose that each vertex $v$ acts elliptically on $T$; in this case the proof is similar to the corresponding case in the proof of \cite[Theorem A]{GH}, but we include the details for the sake of completeness. Since the action is non-trivial, there must be some vertices $v$ and $u$ in $\Gamma$ such that $\fix(v)\cap \fix(u)=\emptyset$. Let $e$ be an edge in the tree $T$ on the path connecting the subtrees $\fix(v)$ and $\fix(u)$. Let $\Lambda$ be the induced subgraph of $\Gamma$ on the set of vertices which fix $e$. Clearly $A(\Lambda)\leq \stab(e)$. Now we will show that that path in $\Gamma$ from $u$ to $v$ must contain a vertex in $\Lambda$. Note that $\Lambda$ may be empty, in which case the proof will show that there is no path from $u$ to $v$, i.e. $\Gamma$ is disconnected. To that end, let $v=v_0, v_1,..., v_n=u$ be the vertices of a path from $v$ to $u$ in $\Gamma$. Note that $n\geq 2$ since Lemma \ref{l:fix intersect} implies that $u$ is not adjacent to $v$ in $\Gamma$. Since each $v_i$ is adjacent to $v_{i+1}$ in $\Gamma$ they must commute as elements of $A(\Gamma)$, and hence by Lemma \ref{l:fix intersect} $\fix(v_i)\cap \fix(v_{i+1})\neq \emptyset$ for $0\leq i\leq n-1$. It follows that there is a path in the tree $T$ from $\fix(v)$ to $\fix(u)$ which is contained in $\bigcup_{i=1}^{n-1} \fix(v_i)$. Since the path in $T$ from $\fix(v)$ to $\fix(u)$ is unique, the edge $e$ must belong to $\fix(v_i)$ for some $i$, and hence $v_i\in \Lambda$. Therefore, $v$ and $u$ are in different connected components of $\Gamma\setminus\Lambda$.

Now suppose that some vertex $v$ acts hyperbolically on $T$. Let $\Lambda$ be the subgraph induced by the set of vertices which act elliptically on $T$ and which fix the axis of $v$ point-wise. Then for any edge $e$ on this axis, $A(\Lambda)\leq\stab(e)$.

Now consider the connected component of $v$ in $\Gamma\setminus\Lambda$. Let $u$ be a vertex in this component, and let $v=v_0, ..., v_n=u$ be a path from $v$ to $u$. Notice that $v_1$ must be hyperbolic, because elliptic elements which commute with $v$ will belong to $\Lambda$ by Lemma \ref{L1}. But then $v$ and $v_1$ have the same axis by Lemma \ref{L1}. Repeating this argument, we get that $v_2,..., u$ are all hyperbolic with the same axis as $v$. But since the action of $A(\Gamma)$ is minimal and $T$ is not a line, there must exist some vertex $w$ which does not set-wise fix the axis of $v$. Hence $v$ and $w$ are in different connected components of $\Gamma\setminus\Lambda$.

\end{proof}

\end{document}